\documentclass[]{article}
\usepackage[english]{babel}

\usepackage{amsmath}
\usepackage{amsthm}
\usepackage{graphicx}
\usepackage[colorlinks=true, allcolors=blue]{hyperref}
\usepackage{amsfonts}
\usepackage{cleveref}
\usepackage{ amssymb }
\usepackage[shortlabels]{enumitem}
\usepackage{tikz}
\usetikzlibrary{positioning,decorations.pathmorphing,shapes}
\tikzset{main node/.style={circle,draw,minimum size=.6cm,inner sep=0pt},}
\tikzset{blob node/.style={circle,draw,minimum size=3cm,inner sep=0pt},}

\usepackage{comment}
\usepackage{subcaption}
\usepackage{algorithm}
\usepackage{algpseudocode}

\usetikzlibrary{positioning}
\usetikzlibrary{calc}
\usetikzlibrary{arrows}
\usetikzlibrary{decorations.markings}
\usepackage{thm-restate}
\usepackage{todonotes}

\newcommand{\cut}{\mathrm{cat}}
\newcommand{\cat}{\cut}
\newtheorem{theorem}{Theorem}[section]
\newtheorem{corollary}[theorem]{Corollary}
\newtheorem{lemma}[theorem]{Lemma}
\newtheorem{definition}[theorem]{Definition}

\crefname{lemma}{Lemma}{Lemmas}
\counterwithin{figure}{section}

\title{Extremal Cat Herding}
\author{Rylo Ashmore, Danny Dyer, Rebecca Milley}

\begin{document}
\maketitle
\begin{abstract}
    The game of Cat Herding is one in which cat and herder players alternate turns, with the evasive cat moving along non-trivial paths between vertices, and the herder deleting single edges from the graph. Eventually the cat cannot move, and the number of edges deleted is the cat number of the graph. We analyze both when the cat is captured quickly, and when the cat evades capture forever, or for an arbitrarily long time. We develop a reduction construction that retains the cat number of the graph, and classify all (reduced) graphs that have cat number 3 or less as a finite set of graphs.  We expand on a logical characterization of infinite Cat Herding on trees to describe all infinite graphs on which the cat can evade capture forever. We also provide a brief characterization of the graphs on which the cat can score arbitrarily high. We conclude by motivating a definition of cat herding ordinals for future research.
\end{abstract}
\section{Introduction}

The game of Cat Herding begins with a graph $G$. Initially, the cat player chooses a vertex $v\in V(G)$ to place their token. Then, alternating play, the herder chooses some edge $e
\in E(G)$ to delete, and the cat chooses some non-trivial path to move along to a new vertex. Play proceeds until the cat is \emph{captured} (is on a vertex of degree $0$), at which point the number of edges deleted is the \emph{score} of the game, sometimes expressed as a number of \emph{points} in the game. The game was introduced in \cite{ashmore2024cutscatscompletegraphs}, where a thorough investigation into $\cat(K_n)$ is made, among other preliminary results. We recall some fundamental results from that paper, as we will use them.

First, the game is monotonic with respect to subgraphs. More precisely, we know the following.
\begin{theorem}\label{subgraphMonotone}\cite{ashmore2024cutscatscompletegraphs}
    If $H$ is a subgraph of $G$, then $\cat(H)\leq \cat(G)$. Further, if $v\in V(H)$, then $\cat(H,v)\leq \cat(G,v)$.
\end{theorem}

This theorem is often useful, as if we observe that a complicated graph $G$ has a simple graph $H$, such as a path or cycle, as a subgraph, we may provide a lower bound on the cat number of the overall graph $G$.

We also use the cat numbers on cycles, paths, and stars, proven in \cite{ashmore2024cutscatscompletegraphs}.

\begin{theorem}\label{thm:cat_numbers_on_paths_ECH}\cite{ashmore2024cutscatscompletegraphs}
    If $n\geq 2$, then $\cut(P_n)=\lceil\log_2n\rceil$.
\end{theorem}
\begin{theorem}\cite{ashmore2024cutscatscompletegraphs}
    If $2k,2k+1\geq 3$, then $\cat(C_{2k+1})=\cat(C_{2k})=\lceil \log_22k\rceil+1$. 
\end{theorem}

Finally, we recall that stars $S_n=K_{1,n-1}$ have cat number $2$ generally, only excepting small $n$.
\begin{theorem}\cite{ashmore2024cutscatscompletegraphs}
    If $n\geq 2$, then $\cat(S_n)=2$. For small stars, $\cat(S_0)=\cat(K_1)=0,\cat(S_1)=\cat(P_2)=1$.
\end{theorem}

In this paper, we will recall the full definition of the game of Cat Herding. We will then look to the extremes of Cat Herding. Namely, we examine graphs of low cat number, up to cat number 3. This follows previous characterizations of graph parameters in small cases, such as characterizing graphs with search number up to $3$ in \cite{graph_searching_complexity}, or fast search number up to $3$ characterized in \cite{DERENIOWSKI20131950}. Similar characterizations are done for graphs of any cop number\cite{CLARKE20121421}, including when $1$ cop suffices to capture the robber (copwin graphs). Noting that the game of Cat Herding does not have a corresponding notion of catwin games, we then look at graphs of infinite and unbounded cat number.

In the first two sections, we define a notion of `pruning' a graph. Pruning graphs aims to make the graphs into smaller graphs of equal cat number. Once we have this tool, we prove that the set of pruned graphs of cat numbers 2 or 3 are finite, and we fully categorize them. Of particular interest to this categorization is that there are exactly three (unpruned) graphs of cat number 2 ($P_2,P_3,C_3$), all pruned trees of cat number 3 are paths, and there is a unique pruned graph of cat number 3 with more than one cycle. We observe there is significant technical difficulty in developing the pruning techniques further, as even pruning leaves is non-trivial.

Finally, we end with an extended analysis of infinite cat herding. By allowing our graphs to be infinite, we allow for the cat to evade capture forever (in particular, the obvious bound of $\cat(G)\leq |E(G)|$ may become less relevant). We say that a graph is \emph{cat-win} if the cat has a strategy to avoid being captured on an isolated vertex forever, and the graph is \emph{herder-win} if the herder can eventually isolate the cat to a single vertex. Note then that $\cat(G)\leq |E(G)|$ implies all finite graphs are herder-win. We find a structural characterization of the cat-win trees. We use this to generalize to all cat-win graphs, without the restriction of trees. Finally, we observe that there are still graphs where the cat will be captured in finite time, but they can score arbitrarily high. We characterize the graphs which enjoy this property, and end with a motivation for defining transfinite ordinal cat numbers on infinite graphs.

\section{Graphs of Cat Number at most 2}
There is established precedent that the cat number of graphs can grow without bound (see complete graphs or paths for example). An alternate question to develop intuition for this problem is to examine and classify all graphs of low cat number. In this section, we only consider connected graphs, because \cite{ashmore2024cutscatscompletegraphs} showed that the herder always has an optimal move in the same component as the cat.

If we were to try and consider graphs of cat number $0$, we observe that our game is somewhat degenerate. If the cat chooses a vertex with degree at least $1$, then some edge must be cut to capture the cat. As such, if $\cat(G,v)=0$, then $d(v)=0$. But if the cat starts on an isolated vertex in a connected graph, then $E(G)=\emptyset$. But since $E(G)=\emptyset$, then the herder has no valid move, and the game ends immediately. To avoid this degenerate case, we begin by analyzing graphs of cat number $1$. It turns out, the only way to have cat number $1$ is for $G\cong P_2\cong K_2$.
\begin{theorem}
    If $G$ is a connected graph, then $\cat(G)=1$ if and only if $G\cong K_2$.
\end{theorem}
\begin{proof}
    Let $G$ be a connected graph. If $G\cong K_2$, then $\cat(K_2)=1$ is trivial. Conversely, if $\cat(G)=1$, then for all vertices $v$, $\cat(G,v)\leq 1$. If $G$ contains a $3$-path, then $\cat(G)\geq 2$ by \Cref{subgraphMonotone}, so $G$ must have no three vertices connected. Thus $|V(G)|\leq 2$, and enumeration verifies the claim.
\end{proof}
Note that a depth-first search (DFS) can detect if the cat number is $1$ in constant time.

\begin{lemma}\label{lem:leafScores1}
    Let $G\ncong K_1$ be a connected graph. For all $v\in V(G)$, $\deg(v)=1$ if and only if $\cat(G,v)=1$.
\end{lemma}
\begin{proof}
    If $v$ has a single edge incident, then the herder cuts it and wins in one move. Any other move leaves the cat a follow up move (along the edge), and so this is the unique optimal move for the herder.

    Alternatively, if $\cat(G,v)=1$, then there is some edge $e$ to cut resulting in the cat's capture at $v$. If $e$ is not incident to $v$, then the cat has a follow-up move. If $e$ is not a unique edge incident to $v$, then the cat has a follow-up move, and $\cat(G,v)\geq 2$. Thus $v$ has a unique edge $e$ incident to $v$, and so $\deg(v)=1$.
\end{proof}

In order to describe graphs of small cat number, we need to understand how to add structures to a graph without changing the cat number. Reversing this process, we obtain which structures can be removed from a graph without changing the cat number. We will then classify the graphs of low cat number which cannot have these removal operations on them.

The first such operation is leaf duplication.
\begin{lemma}
    Let $G$ be a connected graph with vertices $u,v$. If $\deg(G,u)\geq 2$, $uv$ an edge, and $\deg(G,v)=1$, then the graph $G_v$ obtained by adding a new vertex $x$ and edge $ux$ has $\cat(G,a)=\cat(G_v,a)$ for all $a\in V(G)$, and $\cat(G_v,x)=1$. Namely, the cat number does not increase.
\end{lemma}
\begin{proof}
    Let $G,v,G_v$ be as described. Since $G$ is a subgraph of $G_v$, we know that $\cat(G,a)\leq \cat(G_v,a)$ for all vertices $a\in V(G)$ by \Cref{subgraphMonotone}. By construction, we know that $\deg(x)=1$, and so by \Cref{lem:leafScores1}, $\cat(G_v,x)=1$. Let $a\in V(G)$ be an arbitrary vertex. We must show $\cat(G,a)\geq \cat(G_v,a)$.

    Consider a game where both players play optimally in $G$, as a subgraph of $G_v$. We claim neither player will gain by deviating from this strategy.

    We observe that if the herder's unique optimal move in $G$ is to cut $uv$, then the cat is already on $v$, or the cat is on $u$ and $uv$ is the last edge. Suppose alternately, that the cat is on $w\neq v$, and the herder's optimal move is to cut $uv$. Then the cat is restricted to the subgraph currently induced by $V(G)\setminus \{v\}$, say $H$. The cat will go on to score $\cat(H,w')+1$ from their play to $w$. If, alternately, when the cat is on $w$, the herder had chosen any edge $e\in E(H)$ to cut instead, and proceeded to just play in $H$, they would have scored $\cat(H,w')$ also. However, the cat would potentially have access to $v$ on the final turn, and so our actual upper bound is $\cat(H,w')+1$, again (and so the herder does not lose performance with this alternate strategy). As such, a herder cut of $uv$ is only optimal if the cat is already on $v$, or if $E(H)=\emptyset$. Finally, we note that if $E(H)=\emptyset$ and the cat is not on $v$, then the cat is on $u$, since the game has not ended yet. 
    
    The cat can modify any strategy for $G_v$ where they go to $x$ with one where they go to $v$ under the same circumstances (such a move is available since $uv$ is not cut until the cat is on $v$, or it is the final move of the game).
    
    As both result in capture in the next cut, this performs at least as well. As such, without loss of generality, the cat will play on $v$ rather than $x$. The herder only needs to cut $ux$ if the cat goes to $x$ by a similar argument as above with $v$, and so both players will play on $G$ within $G_v$.
\end{proof}

\begin{corollary}
    If $G$ is a graph, let $G'$ be defined by taking all structures $x\sim y\sim z$ with $\deg(x)=\deg(z)=1,\deg(y)\geq 3$, and deleting $z$ (iteratively). Then $\cat(G)=\cat(G')$ and the remaining vertices do not have their cat numbers changed.
\end{corollary}
\begin{proof}
    Let $G$ be such a graph. Consider $G'$ constructed from one such deletion. Then in $G'$ we have that $\deg(y)\geq 2$, $yz$ an edge, and $\deg(G,z)=1$, so the previous result gives us that $G,G'$ have the same cat numbers. By induction we get that any finite number of these operations results in unchanged cat numbers.
\end{proof}

Observe that this is not the same as pruning any pair of leaves sharing a stem to just one leaf, as this would allow pruning $P_3$ to $P_2$, which does change the cat number. We need that the stem also has degree at least $3$.

If $G=G'$ after the pruning process described above (and $G$ is connected), we call $G$ a \emph{pruned graph}. 

We give a classification of vertices of cat number $2$, similar to \Cref{lem:leafScores1}.

\begin{lemma}
    Let $G$ be a connected graph and $v\in V(G)$ with $\cat(G,v)\geq 2$. Then $\cat(G,v)=2$ if and only if there exists an edge $e$ such that $v$ is the center of a star component in $G-e$.
\end{lemma}
\begin{proof}
    ($\Rightarrow$) This is immediate from examination of statement: the herder cuts edge $e$, and the cat must move to a leaf, scoring one more.

    ($\Leftarrow$) If $\cat(G,v)=2$, then the herder must have a move series capturing the cat in at most $2$ cuts, say first deleting edge $e$. After this move, the cat must move to a vertex where they have $\cat(G,x)=1$, and by \Cref{lem:leafScores1}, this must be a leaf. Thus, the vertex $v$ in $G-e$ is at the center of a star component.
\end{proof}

In order to fully classify cat number $2$ graphs, we need to identify forbidden structures, namely, ways to make cat number $3$.

\begin{lemma}
    If $G$ is a connected graph, $v\in V(G)$, and for all $e\in E(G)$, $v$ is an endpoint of a $3$-path or in a triangle in $G-e$, then $\cat(G,v)\geq 3$.
\end{lemma}
\begin{proof}
    Let $G$ be connected and $v\in V(G)$. Suppose the cat is at $v$ and the herder cuts $e$. If the cat is on the endpoint of a $P_3$, then they move to the center. If they are in a $K_3$, they move anywhere in the $K_3$. Either scores at least two more points plus $e$, so they score at least $3$.
\end{proof}

We introduce a common structure allowing the cat to score at least $3$ points.

\begin{definition}
    A \emph{cycle with a tail} is a graph induced by the edges of a cycle $C$ and one edge $e\not\in E(C)$ that meets the cycle's vertices ($V(e)\cap V(C)\neq \emptyset$).
\end{definition}

Observe that this allows cycles with a chord, or a cycle with an edge to a new vertex. It is sometimes useful to consider when the herder's move doesn't matter. In particular, if the herder makes a cut in a different component than the one the cat is in, we call this a \emph{passing move}.

\begin{lemma}\label{lem:circTail}
    If $G$ contains a cycle with a tail, then $\cat(G)\geq 3$, with at least two connected vertices $u,v$ such that $\cat(G,u)=\cat(G,v)=\cat(G)$.
\end{lemma}
\begin{proof}
    Let $G$ contain cycle $C$ and edge $uv$ incident to the cycle, say at $v\in V(C), uv\not\in E(C)$. Then the cat starts anywhere in $V(C)-\{v\}$ (at least two options). If the herder cuts $uv$ (or makes an passing move), then the cat plays on the cycle, scoring at least $2$ more, netting a total of three points. Otherwise, the herder cuts in the cycle, and the cat then moves to $v$, which must be the center of a $3$-path as $\deg(v)\geq 3$ before the single cut. This also scores at least $2$ more, netting a total of $3$ points.
\end{proof}
Now we can classify graphs of cat number $2$.
\begin{figure}
    \centering
    \begin{subfigure}[c]{.32\textwidth}
        \centering
        \begin{tikzpicture}
            \node[main node] (a) at (-1,0) {};
            \node[main node] (b) at (1,0) {};
            \node[main node] (c) at (0,1.75) {};
            \draw (c) -- (b) -- (a) -- (c);
        \end{tikzpicture}        \caption{$C_3$}
        \label{fig:cat2cycle3}
    \end{subfigure}
    \begin{subfigure}[c]{.32\textwidth}
        \centering
        \begin{tikzpicture}
            \node[main node] (a) at (0,-1.75) {};
            \node[main node] (b) at (0,0) {};
            \node[main node] (c) at (0,1.75) {};
            \draw (a) -- (b) -- (c);
        \end{tikzpicture} 
        \caption{$P_3$}
        \label{fig:cat2path3}
    \end{subfigure}
    \begin{subfigure}[c]{.32\textwidth}
        \centering
        \begin{tikzpicture}
           \node[main node] (a) at (0,3) {};
            \node[main node] (b) at (0,1.5) {};
            \node[main node] (c) at (0,0) {};
            \node[main node] (d) at (0,-1.5) {};
           
            \draw (a)--(b)--(c)--(d);
        \end{tikzpicture}
        \caption{$P_4$}
        \label{fig:cat2path4}
    \end{subfigure}
   
    \caption{Types of (pruned) graphs with cat number 2.}
    \label{fig:cat2}
\end{figure}
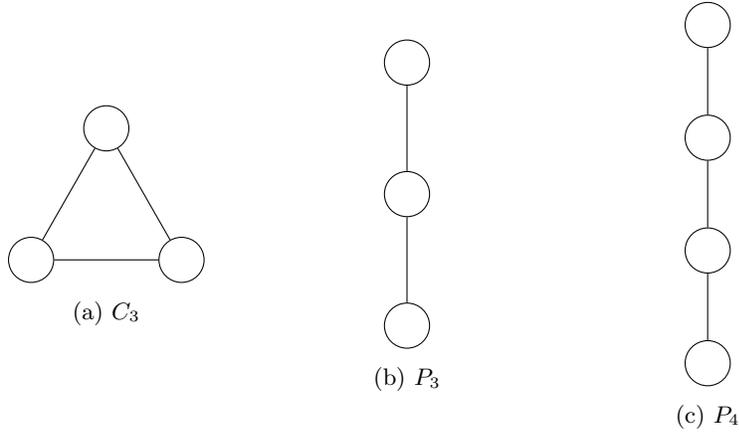
\begin{theorem}
    Let $G$ be a pruned graph. Then $\cat(G)=2$ if and only if $G\cong C_3,P_3,P_4$. 
\end{theorem}
\begin{proof}
    The graphs $G\cong C_3,P_3,P_4$ are shown in \Cref{fig:cat2} for reference. 
    
    ($\Leftarrow$) Suppose that $G$ is a connected pruned graph. It is clear that if $G\cong C_3,P_3,P_4$, then $\cat(G)=2$ by examining each of the cases.

    ($\Rightarrow$) Suppose that $G$ is a connected pruned graph, and that $\cat(G)=2$. Let $V$ be the set of vertices $v$ such that $\cat(G,v)=2$. Note that $V\neq \emptyset$ since $\cat(G)=2$. For any $v\in V$, we have $\cat(G,v)=2$, and there exists edge $e_v$ such that $G-e_v$ has $v$ the center of a star component. Further, there are no vertices $v'$ for which $\cat(G,v')\geq 3$. We now consider cases: either there is some $v\in V$ such that all $e_v$ is non-incident with $v$, or all $v\in V$ have some $e_v$ incident with $v$.

    \begin{enumerate}
        \item If $e_v=xy$ is not incident to $v$ for some $v\in V$, then we fix such a $v$. Then $v$ is the center of a star component in $G-xy$. Since $G$ is connected and $v$ is a center of a star in $G-xy$, at least one of $x,y$ is incident to $v$. Without loss of generality, say $xv$ is the edge.
        \begin{enumerate}
            \item If then $vy$ is an edge, then $vxy$ is a cycle. Further, there can be no other pendant edges to $v,x,y$ (other than $vx,vy,xy$), since \Cref{lem:circTail} would give the graph cat number $3$. Thus if $vy$ is an edge, then $vxy$ forms $C_3$ and there are not other vertices. 
    
            \item Otherwise, $vy$ is not an edge. Since $v$ has cat number $2$, it must have degree at least two, and thus has some other adjacent vertex, say $w$. For every vertex $w\in N[v]$, either $w$ is a leaf in $G$, or $w=x$ or $w=y$, by $v$ being the center of a star component in $G-xy$. We know $y\not\in N[v]$, so $w=x$ or $w$ is a leaf in $G$. For $w\neq x$, duplicates will be pruned by the pruning process, so we know that $w$ is a unique vertex incident to $v$ (other than $x$). Thus $wvxy$ is a path. If any vertex $z\neq x$ is adjacent to $y$, then $wvxyz$ is a path, and $P_5$ has cat number three, violating our assumption. Similarly, if additional leaves are pendant from $x,v$, then they would have been pruned, and adding any more structure beyond a leaf to $x,v$ creates a $5$-path with cat number at least $3$. Thus $d(w)=d(y)=1,d(v)=d(x)=2$, and we have $P_4$.

        \end{enumerate}
    \item Finally, if $e_v=vx$ for all $v\in V$, then let $v,vx$ be such a vertex and edge pair. It follows in $G-vx$ that $v$ is the center of a star. Thus $\deg(v)\geq 2$. If $\deg(v)>2$, say $vw_1,vw_2$ are both edges not meeting $x$, then $v$ the center of a star in $G-vx$ means that $w_1,w_2$ are both leaves, and one should have been pruned. Thus $w$ is the (optional) unique non-$x$ vertex incident to $v$, and $\deg(w)=1$. Now, either $\deg(x)=1$ and $G=wvx$ is $P_3$, or $\deg(x)\geq 2$, and then a symmetric argument on $x$ forces $x$ to have at most one non-$v$ vertex incident, say $y$, which must be a leaf. Then $G=wvxy$, and we get $P_4$.
    \end{enumerate}
    Since in all cases, we find that $G$ is one of our three graphs, we complete this direction of the argument.
\end{proof}
Pruning degree $1$ vertices as described takes time linear in the graph, after which a constant time algorithm can detect if the resulting graph is one of the identified structures.
\section{Graphs of Cat Number 3}
Having determined the graphs of cat number up to $2$, we proceed with cat number $3$. Cat number $3$ graphs are not too much harder. In particular, having multiple cycles in a graph is prone to making the cat number large, and there is only one way to have cat number $3$ with multiple cycles.
\begin{lemma}\label{lem:cat3_2cycles}
    If connected graph $G$ has two distinct cycles $C_1,C_2$ and $\cat(G)=3$, then $|V(G)|=6$, and $G$ is two disjoint $3$-cycles with a bridge between them.
\end{lemma}
\begin{proof}
    Let $G$ be a connected graph with distinct cycles $C_1,C_2$, and $\cat(G)=3$. Suppose the cycles $C_1,C_2$ meet at some vertex (and possibly an edge). Let $v\in V(C_1)\cap V(C_2)$ such that $\deg(v)\geq 3$ in $C_1\cup C_2$, which must exist since $C_1\neq C_2$. Then the cat's strategy is to play on $v$, and after any cut, the cat has a cycle with a tail, which \Cref{lem:circTail} gives us a score of at least $3$. Adding the initial cut after opening on $v$, the game has a score of $4$ points, a contradiction. 

    Thus cycles $C_1,C_2$ share no vertices (and thus no edges). Let $P_{u,v}$ be a path between $C_1,C_2$, say with $u\in V(C_1),v\in V(C_2)$. If either $C_1,C_2$ has $4$ or more vertices, then without loss of generality assume $C_1$ has at least $4$ vertices, and the cat starts there. After cut $e$ is made, if $e\not\in C_1$, the cat plays in $C_1$ with at least $4$ vertices, and attains at least $3$ more points. Otherwise, $e\in C_1$, and the cat can play in $C_2+P_{u,v}$, which must by \Cref{lem:circTail} score at least $3$ more. This contradicts $\cat(G)=3$. Note, this may not be optimal, but all we need to do is show that the cat can evade capture for at least $4$ steps.

    Thus cycles $C_1,C_2$ are disjoint and both of size $3$. If $P_{u,v}$ contains more vertices than $u,v$, say $w$, then the cat can start on $w$, and after any cut, they can reach a $C_3$ with a tail, scoring at least $3$ more by \Cref{lem:circTail}. Thus $P_{u,v}=uv$. Finally, if any edges outside of $C_1,C_2$, or $uv$ exist, say $cx$, for $c\in V(C_1\cup C_2)$, then the cat may start at $c$. Up to symmetry, $c=u$ or $c\neq u,v$. It is not hard to verify in either case that the cat starting at $c$ can reach a cycle with a tail after any cut, attaining a final score of at least $4$.

    In conclusion, the unique graph with two cycles and cat number $3$ is two $3$-cycles with a bridge.
\end{proof}
\begin{corollary}\label{cat3Classification}
    If $\cat(G)=3$, then $G$ is $2C_3+e$, $G$ is a tree, or $G$ contains exactly one cycle.
\end{corollary}
\begin{proof}
    This is immediate from counting cycles in $G$ and applying \Cref{lem:cat3_2cycles} when there are at least $2$ cycles.
\end{proof}
We observe that from the cat numbers on paths and cycles, we obtain cat number $3$ from $P_n$ for $5\leq n\leq 8$ and cycles $C_n$ for $4\leq n\leq 5$, so in fact there are graphs with $0,1,2$ cycles of cat number $3$.

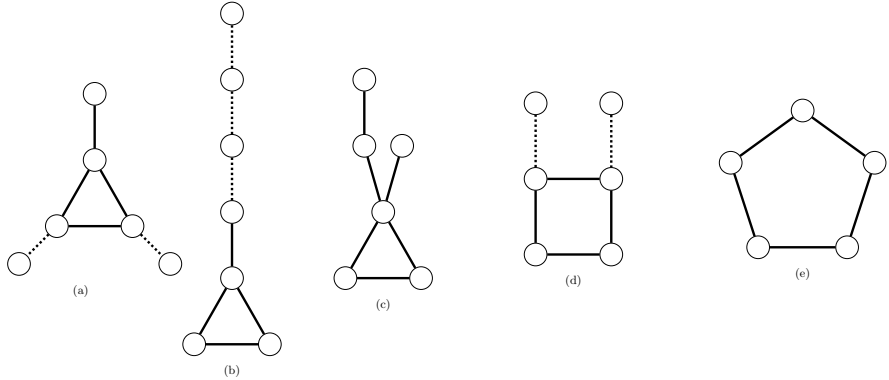
\begin{figure}
 \resizebox{\linewidth}{!}{
    \centering
    \begin{subfigure}[c]{.32\textwidth}
        \centering
        \begin{tikzpicture}
            \node[main node] (a) at (-1,0) {};
            \node[main node] (b) at (1,0) {};
            \node[main node] (c) at (0,1.75) {};
            \node[main node] (d) at (2,-1) {};
            \node[main node] (e) at (-2,-1) {};
            \node[main node] (f) at (0,3.5) {};
            \draw[line width=1.8pt] (f) -- (c) -- (b) -- (a) -- (c);
            \draw[dotted, line width=1.8pt] (b) -- (d);
            \draw[dotted,line width=1.8pt] (a) -- (e);
        \end{tikzpicture}        \caption{}
        \label{fig:cat3cycle3a}
    \end{subfigure}
    \begin{subfigure}[c]{.32\textwidth}
        \centering
        \begin{tikzpicture}
            \node[main node] (a) at (-1,0) {};
            \node[main node] (b) at (1,0) {};
            \node[main node] (c) at (0,1.75) {};
            \node[main node] (d) at (0,3.5) {};
            \node[main node] (e) at (0,5.25) {};
            \node[main node] (f) at (0,7) {};
            \node[main node] (g) at (0,8.75) {};
            \draw[line width=1.8pt] (c) -- (a) -- (b) -- (c) -- (d);
            \draw[dotted, line width=1.8pt] (d)--(e)--(f)--(g);
        \end{tikzpicture} 
        \caption{}
        \label{fig:cat3cycle3b}
    \end{subfigure}
    \begin{subfigure}[c]{.32\textwidth}
        \centering
        \begin{tikzpicture}
           \node[main node] (a) at (-1,0) {};
            \node[main node] (b) at (1,0) {};
            \node[main node] (c) at (0,1.75) {};
            \node[main node] (d) at (-.5,3.5) {};
            \node[main node] (e) at (-.5,5.25) {};
            \node[main node] (f) at (.5,3.5) {};
            \draw[line width=1.8pt] (a)--(b)--(c)--(a);
            \draw[line width=1.8pt] (c)--(d)--(e);
            \draw[line width=1.8pt] (f) -- (c);
        \end{tikzpicture}
        \caption{}
        \label{fig:cat3cycle3c}
    \end{subfigure}
    \begin{subfigure}[c]{.49\textwidth}
        \centering
        \begin{tikzpicture}
            \node[main node] (a) at (-1,-1) {};
            \node[main node] (b) at (-1,1) {};
            \node[main node] (c) at (1,1) {};
            \node[main node] (d) at (1,-1) {};
            \draw[line width=1.8pt] (a)--(b)--(c)--(d)--(a);
            \node[main node] (e) at (1,3) {};
            \node[main node] (f) at (-1,3) {};
            \draw[dotted,line width=1.8pt] (e)--(c);
            \draw[dotted,line width=1.8pt] (b)--(f);
        \end{tikzpicture} 
        \caption{}
        \label{fig:cat3cycle4}
    \end{subfigure}
    \begin{subfigure}[c]{.49\textwidth}
        \centering
        \begin{tikzpicture}
            \foreach\angle in {90,162,234,306,378}{
            \node[main node] (a\angle) at (\angle:2) {};}
            \draw[line width=1.8pt] (a90) -- (a162) -- (a234) -- (a306) -- (a378) -- (a90);
        \end{tikzpicture} 
        \caption{}
        \label{fig:cat3cycle5}
    \end{subfigure}}
    \caption{Types of (pruned) graphs with one cycle and cat number 3. Any dotted edges are optional, provided the resulting edges induce a connected graph. They generally fall into five categories, with (a)-(c) being Triangle-type (a)-(c), (d) being Square-type, and (e) being Pentagon-type.}
    \label{fig:cat3cycle}
\end{figure}
Following \Cref{cat3Classification}, we have a natural way to continue the investigation of cat number 3 graphs, namely, those with cycles, and those that are trees. We claim that all the pruned graphs with cat number $3$ and exactly one cycle appear in \Cref{fig:cat3cycle}.
\begin{theorem}\label{thm:cat_number_3_1_cycle}
    Let $G$ be a pruned graph with exactly one cycle. Then $\cat(G)=3$ if and only if $G$ is isomorphic to $C_3$ with $1$ to $3$ tail leaves, $C_3$ with a single tail path of $1-4$ vertices, $C_3$ with two tail paths of length $1,2$ incident to the same cycle vertex, a square $C_4$ with $0$, $1$, or $2$ non-diametrically opposed leaves, or $C_5$. Namely, $G$ has cat number $3$ if and only if it appears in \Cref{fig:cat3cycle}.
\end{theorem}
\begin{proof}
    Let $G$ have exactly one cycle $C$, have $\cat(G)=3$, and be pruned. If $C$ has $6$ or more vertices, then $\cat(G)\geq 4$, since the cat could just play on the cycle. So we can split based on whether $C$ has $3,4$, or $5$ vertices.
    \begin{enumerate}
        \item Suppose $C$ has $3$ vertices. If $G=C$, then $\cat(G)=\cat(C_3)=2$, a contradiction. Thus there is some other structure, and since $C_3=K_3$, that must be an added edge $cx$. Let $E$ be the edges added to $C_3+cx$ in order to obtain $G$.
        \begin{enumerate}
            \item If all edges in $E$ meet $C$, then after pruning redundant leaves, it will be of type triangle (a). It remains to show that all triangle-type (a) all have cat number 3. We have already shown that a cycle with a tail achieves at least 3 for a lower bound. Let $G$ be the triangle with an edge off of every triangle vertex. If the cat starts on a leaf, they lose in one cut. Otherwise, the cat starts on the cycle, and the herder cuts the edge diametrically opposed to the cat on the cycle. If the cat moves to a leaf, the cat loses in two cuts. Otherwise, the cat moves to one of the cycle vertices incident to the cut edge. The herder then cuts the other incident cycle edge, and the cat must move to the leaf, resulting in the score of $3$. Since all graphs of this type are a subgraph of the graph considered, this suffices.
            \item Suppose alternately that $E$ meets $C$ at at least two vertices, say $xc_1,c_2y$. Because we are not in the first case, $G$ also has edge $yz$ not meeting $\{c_1,c_2\}$. Then the cat starts at $c_3$. If the herder cuts edge $e$ not part of the cycle, at least one tail from the cycle is still intact ($xc_1,c_2y$), and so the cat can score at least $3$ more. Thus the herder cuts in the cycle. The paths $xc_1c_2yz$ and $xc_1c_3c_2yz$ are both paths of at least $5$ vertices that share no cycle edges, so at least one of which must be intact. Thus the cat can move to $c_2$ and score at least $3$ more, as they are in the center of a $P_5$. Thus the cat can force a score of at least $4$, and this case is not possible. 
            \item Now we know that that not all edges in $E$ meet $C$, and $E\cap C=\{x\}$. Further suppose $E$ forms a path ending at $x$, then the graph reaching $C_3$ with a $5$-vertex $x_1\dots x_5$ tail attached with a bridge $cx_1$ leads to the cat starting on $x_1$, and on any cut, the cat either has access to $P_5$ or $C_3+e$, both scoring $3$ recursively. Thus $E$ forming a path has a maximum length as drawn in (b). Note that the $x_1\dots x_4$ path will result in a cat number of $3$ since if the cat starts in the triangle, the herder cuts off the triangle, and scores $1+\cat(K_3)=3$. Otherwise, the cat starts on the path, and the herder cuts off the triangle, and scores $1+\cat(P_4)=3$. Again, the lower bound of $3$ comes from $C_3+e$ having cat number $3$.
            \item  \label{case:broken_p2_duplicate}Lastly, we have all edges $E\cap C=\{x\}$, there are edges that don't meet $C$, and $E$ does not induce a path. If at least $2$ $P_2$s are incident with $x=c_1$, then the cat starts at $c_2$ or $c_3$. After any cut not in the triangle, there must be $C_3+e$ left over, scoring at least $3$ more. Thus any cut to keep score $3$ must be in the triangle. The cat can then move to $x=c_1$, the center of a $P_5$, and again scores at least $3$ more. Thus there can be at most one path of length at least $2$ off of $x$. If there is no path of length at least $2$ off of $x$, then after pruning, we have $C_3+e$ again. So we have exactly one path of length $2$ off of $x$. Since $E$ does not induce a path, there must be other pendants. If the path of length $2$ also branched, it would be pruned. Thus there is an additional pendant on $x$. Multiple leaves will be trimmed again, so without loss of generality, there is exactly one leaf incident to $x$ after pruning. If there is at least $P_3$ and a leaf off of $x$, say $G$ contains cycle $c_1c_2c_3$ with edge $c_1y$, and path $c_1x_1x_2x_3$, all distinct vertices. Then the cat starts on $c_2$ again. After any cut not in the cycle, there is a $C_3+e$ again, and the cat scores at least $3$ more. After a cut in the cycle, $x_3x_2x_1c_1y$ is a $P_5$ with $x_1\neq c_2$ in the center. This again scores at least $3$ more, and we see that triangle-type (c) is the unique case here. To see that triangle-type (c) scores $3$, we note that a triangle with a tail provides a lower bound. For a herder strategy, the herder cuts $c_2c_3$, and then the cat is on $x_1,c_1$, or a leaf. A leaf results in a score of $2$ for the herder, so the other case results in the herder cutting $x_1c_1$ and then the cat must move to a leaf, scoring $3$ total. 
        \end{enumerate}
        \item Suppose $C$ has $4$ vertices. Let $E=E(G)\setminus E(C)$.
        \begin{enumerate}
            \item If $E\cap C=\emptyset$, then $G=C_4$, and this is in our diagram.
            \item If $E\cap C=\{c_1\}$, then we further suppose that $E$ contains an edge not incident to $C$. Then by connectedness $G$ must contain $x_2x_1c_1$ in addition to cycle $c_1c_2c_3c_4c_1$. Then the cat starts anywhere in the cycle not equal to $c_1$, and after any cut out of the cycle, $C_4$ scores $3$ more, and the cat scores at least $4$. On any cut in the cycle, both $x_2x_1c_1c_2c_3$ and $x_2x_1c_1c_4c_3$ are $P_5$ with edges only shared outside the cycle, so at least one must be intact after the cycle cut. The cat moves to $c_1$ and scores at least $3$ more. Thus if $E\cap C=\{c_1\}$, then every edge incident to $c_1$ is an edge to a leaf. After pruning redundant leaves, we conclude without loss of generality that there is exactly one.
            \item Suppose $E$ meets $C$ at two diametrically opposed vertices, say $c_1,c_3$. Then the cat starts at $c_1$, and after a non-cycle cut, the cat scores $\cat(C_4)=3$ more. Otherwise, a cycle edge is cut, and $c_2,c_4$ are centers of $P_5$ with all edges shared not in the cycle. Again, the cat can move to whichever one is intact and score at least $3$ more. Thus $E$ does not meet $C$ at two diametrically opposed vertices.
            \item Suppose $|E\cap C|=2$. Then $E$ meets $C$ at adjacent vertices, say $c_1,c_2$. We have already argued these cannot extend further, lest the cat scores more than $3$. Again, after pruning, we get the result desired.
            \item To show that all of these have cat number $3$, we note that $\cat(C_4)=3$ for a lower bound, and provide a herder strategy to capture the cat in at most $3$ moves on any of these graphs. First, the herder cuts a cycle edge that is not adjacent to any non-cycle edges. After this, at most two (adjacent) vertices may not be leaves. If the cat is on a leaf after their next move, the herder captures them in two moves. Alternately, the herder cuts the edge between the two non-leaves (or makes a passing move if there is only one non-leaf), and the cat must move to a leaf, leaving the herder to capture on the next (third) move.
        \end{enumerate}
        \item Suppose $C$ has $5$ vertices. If any edge $c_1x$ with $c_1\in V(C),c_1x\not\in E(C)$ exists, then the cat may play on $c_3$. Consider the paths of five vertices $xc_1c_2c_3c_4$, $xc_1c_5c_4c_3$, and $c_4c_5c_1c_2c_3$. For any edge deleted, at least one of these $P_5$s will remain intact, with centers $c_1,c_2,c_5$, all of which are legal moves from $c_3$. Thus the cat will again score at least $3$ more. We have already shown $\cat(C_5)=3$.
    \end{enumerate}
\end{proof}
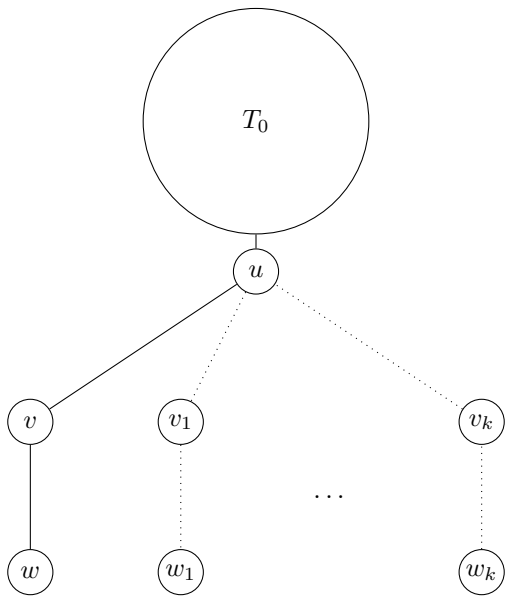
\begin{figure}
    \centering
    \begin{tikzpicture}
        \node[blob node] (r) at (0,3) {$T_0$};
        \node[main node] (u) at (0,1) {$u$};
        \node[main node] (v) at (-3,-1) {$v$};
        \node[main node] (w) at (-3,-3) {$w$};
        \node[main node] (v1) at (-1,-1) {$v_1$};
        \node[main node] (w1) at (-1,-3) {$w_1$};
        \node[] (d) at (1,-2) {$\dots$};
        \node[main node] (vk) at (3,-1) {$v_k$};
        \node[main node] (wk) at (3,-3) {$w_k$};
        \draw (r) -- (u) -- (v) -- (w);
        \draw[dotted] (wk) -- (vk) -- (u) -- (v1) -- (w1);
    \end{tikzpicture}
    \caption{Drawing of $T,T_k$ as in \Cref{lem:treeAddPath2} with edges of $T$ solid and edges of $T_k$ dotted}
    \label{fig:prune2}
\end{figure}

It turns out not only can leaves be duplicated, but also paths of length $2$ can be duplicated on trees when sufficient conditions are met, as we show in \Cref{lem:treeAddPath2}. Conversely, this allows us to prune these structures without changing the cat number. We will discuss the difficulty in generalizing this flavor of result later. 

Given a tree $T$ with named vertices $u,v,w\in V(T)$ such that $uvw$ is a subpath of the tree, and $\deg(u)=\deg(v)=2,\deg(w)=1$, we define a construction $T_k$ as follows. We set $T_k$ to be the tree on vertices $V(T)\cup \{v_i,w_i|1\leq i\leq k\}$ and edges $E(T)$ plus the edges needed to make subpaths $uv_iw_i$ for all $1\leq i\leq k$. 

\begin{lemma}\label{lem:treeAddPath2}
    Suppose $T\ncong P_4$ is a tree with named vertices $u,v,w\in V(T)$ such that $uvw$ is a path with $\deg(T,u)=2$, $\deg(T,v)=2,\deg(T,w)=1$, and let $T_k$ be constructed as above. Then $\cat(T_k,a)=\cat(T,a)$ for all $a\in V(T)$, $\cat(T_k,v_i)=2$, and $\cat(T_k,w_i)=1$. Namely, the cat number does not increase. 
\end{lemma}
\begin{proof}
    Let $T,u,v,w,v_i,w_i,T_k$ be as described. Denote the component of $G-u$ not containing $v,w$ as $T_0$, as diagrammed in \Cref{fig:prune2}. By construction, we know that $\deg(w_i)=1$, so $\cat(T,w_i)=1$. Similarly, with the cat starting on $v_i$, the herder cuts $uv_i$ and wins in two moves. Since $T$ is a subgraph of $T_k$, we know that $\cat(T,a)\leq \cat(T_k,a)$ for all $a\in V(T)$. We must show $\cat(T,a)\geq \cat(T_k,a)$. First we show that in a game on $T$, the cat is on $v,w$ if and only if the next herder cut is $uv,vw$, respectively.
    
    Suppose a game of cat herding is played on $T$. Suppose further (for contradiction) that at some point in play, the cat is on $u$, and the unique optimal herder move is to cut $uv$. Note that an alternative for the herder is to cut the edge between $u$ and $T_0$, resulting in a series of moves where the cat moves to $v$, the herder cuts $uv$, the cat moves to $w$, and the herder cuts $vw$, finally giving a score of $3$ more points from the position when the cat is on $u$ and the herder may choose a move. Thus by cutting $uv$ instead, the herder must be capturing the cat within at most two more cuts. This means the cat must move to a leaf, and thus $T_0$ is a single vertex. This violates the assumption that $T\ncong P_4$, and thus if the cat is on $u$, we may without loss of generality assume the herder does not need to cut $uv$. It follows that if the herder cuts $uv$ in a game on $T$, then the cat is on $v$. Conversely, if the cat is ever on $v$, then the herder could cut $uv$ and win in two cuts. If this is not optimal, then cutting $vw$ must be optimal (since $T$ is a tree and monotonicity from \Cref{subgraphMonotone} implies the herder should cut off as much tree as possible), and $uv$ has already been cut. But when $uv$ was cut, the cat must have been on $v$, and must have moved to $w$, contradicting that the cat is on $v$. Thus the cat is on $v$ if and only if the herder's unique optimal move is to cut $uv$ when playing on $T$. We also know that the cat is on $w$ if and only if the herder's unique optimal move is to cut $vw$ since $\cat(T,w)=1$. It follows that the cat will not move to $v,w$ until capture is guaranteed within $2$ more cuts, and thus the edges $uv,vw$ must be intact until the last two cuts.

    Let $a\in V(T)$. To show $\cat(T,a)\geq \cat(T_k,a)$, we provide a herder strategy on $T_k$ when the cat starts at $a$ (in which they perform at least as well as they would on $T$). For the first $\cat(T,a)-2$ moves, either the cat plays in $T$ or in some of the new vertices. If the cat plays in the new vertices, then either the cat is on $v_i$, and the herder cuts $v_iu$, or the cat is on a $w_i$ leaf, and the herder cuts $v_iw_i$. Either of these cases results in capture in at most $\cat(T,a)$ steps. Otherwise, the cat is playing in $T$ for the first $\cat(T,a)-2$ moves, and the herder plays an optimal move on $T$ for each of them. After these $\cat(T,a)-2$ cuts, note that $uv,uv_i,vw,v_iw_i$ are all intact for all $i$. If the cat moves to any $v_i$, then the herder cuts $uv_i$. The cat will move to $w_i$, and the herder will capture the cat in $\cat(T,a)$ moves. If the cat moves to $w_i$, the herder can delete $v_iw_i$ and win in $\cat(T,a)-1$ cuts. Otherwise, the cat is in $T$ still. Again, the herder follows an optimal $\cat(T,a)$ strategy to make the $(\cat(T,a)-1)$th cut. In $T$, then, the cat must move to a leaf, which the herder may cut. The only alternative in $T_k$ is for the cat to move to either $v_i$ or $w_i$. The $w_i$ are also leaves, and so will result in the cat being captured on the next turn. Assuming that the cat can score better in $T_k$, for contradiction, we conclude that the unique cat move must be to move to $v_i$. Could the cat alternately move to $v$, to stay in $T$? We have already argued that the edge $uv$ is not cut in $T$ unless the cat is on $v$, and our herder strategy only prescribes moves in $T$, and not $T'$ up until now. As such, if the cat can move to $v_i$, then the subpath $vuv_i$ is still intact, and the cat could alternately move to $v$. But $v$ is also not a leaf, and the cat could have delayed capture $2$ turns just in $T$, violating that the herder's moves were optimal in $T$, as the herder now takes $\cat(T,a)-1+2=\cat(T,a)+1$ moves to capture the cat in $T$. As such, neither a move to $v$ nor $v_i$ should be possible as this step, and the cat must be forced into moving to a leaf in $T_k$ after $\cat(T,a)-1$ cuts, as needed.
    
\end{proof}

We can invert this process, instead of adding subpaths, we can view this as pruning them. Given a tree $T$, we define $T'$ to be the \emph{pruned tree of $T$} when we identify all systems of paths $tuv_iw_i$ for $0\leq i\leq k$ such that $\deg(t)\geq 2,\deg(u)=k+2,\deg(v_i)=2,\deg(w_i)=1$, and delete the vertices $\{v_i,w_i|1\leq i\leq k\}$ (iteratively).
\begin{corollary}\label{cor:treePrune1}
    If $T$ is a tree, let $T'$ be the pruned tree of $T$. Then $\cat(T)=\cat(T')$ and for all $v\in V(T'), \cat(T,v)=\cat(T',v)$ (the remaining vertices have the same cat numbers).
\end{corollary}
\begin{proof}
    Let $T$ be such a tree. Then in $T'$ we have that $tuv_0w_0$ is a path, and $\deg_{T'}(t)\geq 2$. Thus $T'\ncong P_4$ is a tree. After pruning $k$ edges ($uv_i$ for $1\leq i\leq k$) incident to $u$, we have $\deg_{T'}(u)=2$. Further, $\deg_{T'}(v_0)=2,\deg_{T'}(w_0)=1$ by construction. Then $T'_k=T$ and \Cref{lem:treeAddPath2} applies, giving the result.

\end{proof}
\begin{corollary}\label{cor:treePrune2}
    If $T$ is a tree, let $T'$ be the pruned tree of $T$, and then iteratively delete all leaves $l$ that have a corresponding path $wvul$ with $\deg(l)=\deg(w)=1$, $\deg(v)=2$, $\deg(u)=3$ with edge $ut$ and $\deg(t)\geq 2$. Then $\cat(T)=\cat(T')$.
\end{corollary}
\begin{proof}
    Let $T''$ be constructed by taking the corresponding leaves $l$ and adding a new vertex $v_l$ that is incident only to $l$. Then after pruning $T''$ to $T'$, vertices $v_l,l$ will be deleted, and result in tree $T'$ with $\cat(T')=\cat(T'')$. Thus $\cat(T')\leq \cat(T)\leq \cat(T'')=\cat(T')$ by noting that $T'$ is a subtree of $T$, which is a subtree of $T''$. It follows that $\cat(T)=\cat(T')$. 
\end{proof}
Note that the argument of \Cref{cor:treePrune2} does not work on \Cref{thm:cat_number_3_1_cycle} in case \ref{case:broken_p2_duplicate}, as extending the leaf to a full $P_2$ \emph{does} in fact increase the cat number. This highlights the necessity of the hypothesis in these pruning lemmas, as we use both that the graph is a tree and that the root of the $P_2$ has degree $2$ after pruning. 

If $T=T'$ after the process described in corollaries \cref{cor:treePrune1} and \Cref{cor:treePrune2}, then we call $T$ a \emph{pruned tree}.

\begin{lemma}\label{lem:cat3treepath}
    If $T$ is a pruned tree with $\cat(T)=3$, then $T$ is a path.
\end{lemma}
\begin{proof}
    Let $T$ be a pruned tree with $\cat(T)=3$. Suppose for contradiction that for some $u\in V(T)$, $\deg(u)\geq 3$. Let $T_1,T_2,T_3$ be the three largest trees of $T-u$ with $|V(T_1)|\geq |V(T_2)|\geq |V(T_3)|$. 

    If $|V(T_2)|\geq 3$, then $|V(T_1)|\geq 3$. If either $T_1,T_2$ are only depth $2$ or less (when inheriting a root from rooting the original tree at $u$), then they should have been pruned to $2$ or fewer vertices. Thus $T_1,T_2$ have depth at least $3$, and so $T$ must contain path $x_3x_2x_1uy_1y_2y_3$ where the $x_i\in V(T_1),y_i\in V(T_2)$ and edge $uv$ with $v\in V(T_3)$. If the cat starts on $u$, then after any cut, at least one of  $x_3x_2x_1uv$, $vuy_1y_2y_3$, or $x_3x_2x_1uy_1$ is intact, and a move to either of $x_1,y_1$ are both legal from $u$. Since $\cat(P_5)=3$ when the cat can start in the center, we find that the cat can score at least $4$, contradicting $\cat(T)=3$.

    Thus $|V(T_2)|\leq 2$ and $|V(T_i)|\leq 2$ for all $i\geq 3$. If $|V(T_1)|\geq 2$, then \Cref{cor:treePrune2} applies and $T_i$ with $i\geq 3$ should have been pruned, contradicting that $T$ is a pruned tree. 

    It follows that $|V(T_1)|< 2$, and so all branches have size at most $1$. But this implies $T$ is a star, contradicting $\cat(T)=3$. This completes the contradiction.
\end{proof}
\begin{corollary}\label{cor:cat_number_3_tree}
    If $T$ is a pruned tree with cat number $3$, then $T\cong P_n$ with $5\leq n\leq 8$.
\end{corollary}
\begin{proof}
    If $T$ is a pruned tree, then by \Cref{lem:cat3treepath} implies $T$ is a path. We know that $\cat(P_n)=3$ if and only if $5\leq n\leq 8$ by \Cref{thm:cat_numbers_on_paths_ECH}, completing the argument.
\end{proof}
\begin{theorem}
    Let $G$ be a pruned graph (pruned tree if $G$ is a tree), Then $\cat(G)=3$ if and only if $G$ is $2C_3+e$, $P_n$ with $5\leq n\leq 8$, or $G$ appears in \Cref{fig:cat3cycle}. 
\end{theorem}
\begin{proof}
    This follows by consolidating \Cref{cat3Classification}, \Cref{thm:cat_number_3_1_cycle}, and \Cref{cor:cat_number_3_tree}.
\end{proof}
Running depth-first search (DFS) to identify all leaves and filtering those with parents of degree $2$, then seeking those that share a parent vertex can all be done in linear time. Again, once the linear-time filtering is done, we have a finite set of graphs, and so an enumeration of vertices allows us to discount any large graphs easily. Finally, it takes at most $8!$ vertex permutations to check isomorphism (in linear time) with a fixed graph in our collection. We again find that the complexity is linear in the size of the graph.

Note that pruning is not easy to generalize to larger classes of structures that enable pruning, even in the case of pruning leaves. In particular, let $\mathcal{P}$ be a multiset of positive integers. Then $S(\mathcal{P})$ is a spider with root $r$ and for every $k\in \mathcal{P}$ with multiplicity $d$, $k$ new vertices are introduced as a path to be attached to $r$ at an endpoint, $d$ times. Then $\mathcal{P}=\{2,2,1\}$ and $\mathcal{P}=\{4,4,1\}$ yield spiders which pruning the leaf branch does not affect the cat number, but in the case of $\mathcal{P}=\{3,3,1\}$, the corresponding spider cannot prune the leaf branch without changing the cat number.
\section{Infinite cat herding}
The converse problem of finding the graphs of low cat number may be viewed as finding graphs for which we have infinite cat number. Since $\cat(G)\leq E(G)$, we must then have infinite graphs. Recall that we say that an infinite graph is \emph{cat-win} when the cat has a strategy which never allows it to be captured, and \emph{herder-win} if the herder can eventually capture the cat. We will also denote a vertex as \emph{cat-win} if a cat placed at that vertex may evade capture forever (under optimal play, with the herder cutting an edge next). We begin with a brief analysis of infinite trees. For convenience, we will denote the infinite complete binary tree as $B^{\omega}$. Formally, this is the graph on the vertex set indexed by $\{L,R\}^*$ (that is, the set of all strings on the alphabet $\{L,R\}$, with $\epsilon$ denoting the empty string) with edges between vertices indexed by $\alpha,\alpha t$ for all $\alpha\in\{L,R\}^*,t\in \{L,R\}$.

We first show a lemma that if the cat has a winning strategy, it can't just be in one direction.

\begin{lemma}\label{lem:inf_tree_cat_win_two_dirs}
    Let $T$ be a rooted tree and $v\in V(T)$. If the cat has a winning strategy from vertex $v\in V(T)$ when they are restricted to playing on descendants of $v$, then there must be two rays of descendants of $v$ with infinitely many cat-win vertices.
\end{lemma}
\begin{proof}
    Note that since the graph is cat-win, there must be some sequence of vertices from which play may continue indefinitely. Then the set of reachable cat-win descendant vertices must contain infinitely many vertices on some ray (consider the herder strategy of always cutting the last edge taken by the cat --- the cat must move along some ray under this strategy). Suppose only one ray (rooted at $v$) contains infinitely many cat-win vertices. Then the cat may (without loss of generality) make some optimal move along this ray, after which the herder can cut just after the cat's location on the ray. The cat may continue with any moves, but the herder now has a winning strategy of cutting the most recently taken edge. When the herder executes this strategy, the cat must be forced along some ray, but all rays by assumption eventually have no cat-win vertices. A herder-win strategy is a contradiction, so there must be at least two rays containing infinitely many cat-win vertices.
\end{proof}

With this lemma, we can show that a tree is cat-win if and only if it contains an infinite complete binary tree minor.

\begin{theorem}\label{thm:infinite_trees}
    A tree $T$ is cat-win if and only if \hspace{.05cm}$T$ \hspace{-.05cm}contains an infinite complete binary tree minor.
\end{theorem}
\begin{proof}
    ($\Leftarrow$) Supposing there is an infinite complete binary tree minor; we claim that a cat's winning strategy is to always play to the root of an infinite complete binary tree. Such a tree must exist initially by assumption, and after any cut $e$, either $e$ is not in the current infinite complete binary tree, or it is cleanly in the left subtree or the right subtree. The cat moves to the alternate side to a new root of a infinite complete binary tree minor, and inductively can survive forever.
    
    ($\Rightarrow$) Suppose $T$ is cat-win. Let $v_\epsilon$ be an cat-win starting location. We think of this as the root of the tree and vertices below as ancestors. We aim to take a given $v_\alpha$ and find $v_{\alpha L},v_{\alpha R}$ ancestors, so that the set of all $v_\alpha$ form a subgraph that can be contracted to $B^\omega$. For the first move, the herder makes any arbitrary passing move.
    
    Since we may interpret $T$ as rooted at $v_{\epsilon}$, we note that the cat is restricted to play on ancestors of their given vertex. Given the cat on $v_\alpha$ with cat play restricted to ancestors of $v_\alpha$ (we will show this can inductively be guaranteed), and so we may use our \Cref{lem:inf_tree_cat_win_two_dirs}.
    
    By \Cref{lem:inf_tree_cat_win_two_dirs}, there must be two distinct rays with infinitely many cat-win vertices on them among the descendant vertices. Take any two of these rays and take vertices that distinguish the rays, say $v_{\alpha L},v_{\alpha R}$, and assume the cat plays two separate (winning) games, one moving to each of these. The herder then may cut (in each game) the parent edge of $v_{\alpha L},v_{\alpha R}$ (respectively), enforcing the inductive constraint of play on ancestors only. 
    
    This construction applies recursively, giving a family of $v_\alpha$ for every left/right string $\alpha\in \{L,R\}^*$ (the set of words on alphabet $\{L,R\}$). Note that if $\alpha\neq \beta$, we have that $v_\alpha\neq v_\beta$ since $T$ is a tree, so we do in fact have uncountably many rays. These vertices form an infinite complete binary tree that is subdivided to some number of paths, which necessarily must be a subgraph of $T$ which can be contracted to $B^\omega$. Thus $T$ contains a $B^\omega$ minor.
\end{proof}

We observe that this provides a description of when the cat wins on trees. We leave a more detailed analysis specifying the herder's strategy, bounding the cat number of herder-win graphs, and connecting formal logic and ray enumeration for a future paper. Here, we will reduce the general problem of cat-win graphs to that of cat-win trees, thus fully solving which infinite graphs are cat-win.

\begin{definition}
A graph is $k$-edge-connected if and only if any pair of vertices has at least $k$ edge-disjoint finite paths between them. This is equivalent to the statement that any disconnecting set of edges is of size at least $k$.
\end{definition}

The above definition may be contested based on whether one believes in the existence of infinite cycles (and as a consequence infinite paths with endpoints). We do not entertain infinite-distance connectivity here, for the purposes of cat-herding, as it makes sense to allow cats to move arbitrarily fast, but not infinitely fast. This also enables use of Menger's theorem for infinite graphs, as proved in \cite{mengerInfinite}.

\begin{lemma}\label{lem:infConnCatWin}
    If $G$ is infinite and $2$-edge-connected, then $G$ is cat-win.
\end{lemma}

\begin{proof}
    Suppose $G$ is infinite and $2$-edge-connected. The cat's strategy is to start on any vertex. After any cut $e$, the cat selects a subgraph $G'$ that is $2$-edge-connected and contains infinitely many vertices, and voluntarily restricts their play to this graph. Since this graph also satisfies the theorem conditions, this strategy may continue indefinitely. The main question for this strategy is if such a $2$-edge-connected subgraph exists after any edge cut $e$.
    
    For sake of contradiction, suppose $G$ is a $2$-edge-connected infinite graph, and suppose edge $e=ab$ is cut such that $G-e$ contains no $2$-edge-connected subgraph with infinitely many vertices. Observe that $G$ is still connected, as our assumption was that $G$ was $2$-edge-connected initially. Let $G'$ be the auxiliary graph where each vertex is a maximally $2$-edge-connected component of $G-e$. For convenience, we will denote the maximal $2$-edge-connected component of $G-e$ that contains vertex $v$ as $G_v$ (note that $2$-edge-connected components' labels are not unique). In $G'$, we have edges between components $G_u,G_v$ when there is an edge $uv\in E(G)$ with $u\in V(G_u),v\in V(G_v)$. If $|V(G')|$ is finite, this gives a $2$-edge-connected infinite subgraph as one of the components by the pigeonhole principle. Thus $|V(G')|$ is infinite. Further, $G'$ is a tree, since it is not $2$-edge-connected. On deleting edge $e=ab\in E(G)$, which component vertices of $G'$ are leaves? Suppose that $G_c$ is a leaf of $G'$. If $G_a,G_b,G_c$ are all disjoint, then let $e'$ be the edge that disconnects $G_c$. Note then that the corresponding $e'\in E(G)$ must also be a cut-edge, and so $G$ is not two-edge-connected, a contradiction. Thus $G_a,G_b$ are the unique leaves,  It follows that $G'$ has at most $2$ leaves and infinite vertices. 
    
    Assume that $G'$ has a vertex of degree at least $3$, say $G_c$. Let $G_x,G_y,G_z$ be neighbours of $G_c$. Since $G$ is $2$-edge-connected, by Menger's Theorem \cite{mengerInfinite}, we obtain $2$ edge-disjoint $a,c$-paths in $G$. One of these paths took $e$, and thus passed through $G_b$, and exactly one of $G_x,G_y,G_z$. But in $G-ab$, $G'$ is a tree, and so there is exactly one $a,c$-path, passing through at most one of $G_x,G_y,G_z$. But then the last of $G_x,G_y,G_z$, say $G_z$, must have the edge $e'$ between $G_c,G_z$ a cut-edge in $G$. Thus no $G_z$ can exist, and so all vertices of $G'$ have degree at most $2$. It follows that $G'$ is a tree with at most $2$ leaves, maximum degree $2$, and infinite vertices. Since $G'$ is a tree, there is exactly one $G_a,G_b$ path in $G'$. If $G_a$ has degree at least $2$, then one edge is not part of the $G_a,G_b$ path. This edge would also be a cut-edge in $G$ then, and so $G_a,G_b$ are the unique leaves of $G'$. Thus $G'$ is an infinite tree with exactly two leaves and maximum degree $2$. Essentially, we have shown we have found an infinite path, which we assumed does not exist.
    
    Thus such $G,e$ could exist, and recursively, we can always find a `smaller' set of which to play on.
\end{proof}

\begin{lemma}\label{lem:infHerderWin}
    If $G$ has no infinite $2$-edge-connected subgraph and $G$ has no infinite complete binary tree minor, then $G$ is herder-win.
\end{lemma}
\begin{proof}
    Let $G$ be a graph with no infinite $2$-edge-connected subgraph. Let $G'$ be the graph obtained by contracting all maximal $2$-edge-connected subgraphs to a point. Since there is no infinite $2$-edge-connected subgraph in $G$, each vertex in $G'$ corresponds to finitely many vertices of $G$. Thus if $G'$ is finite, then $G$ is finite, and $G$ is herder-win. Suppose alternately that $G'$ is infinite. Then $G'$ is an infinite tree, which must contain no infinite complete binary tree minor. By \Cref{thm:infinite_trees}, we know that $G'$ is herder-win. We also know that any finite graph is herder-win, as the herder can delete all edges to win. The herder-win strategy on $G$, then, is to start by playing optimally on $G'$, and on any cat move within the same $2$-edge-connected component, the herder responds by cutting any edge in that component, and on any move in $G'$, the herder responds optimally in $G'$. Since both games are herder-win, the herder must eventually win.
\end{proof}

We are now ready to establish a full characterization of the cat-win graphs.

\begin{theorem}
    If $G$ is any graph, then $G$ is cat-win if and only if at least one of the following hold:
    \begin{enumerate}
        \item $G$ has an infinite complete binary tree minor.
        \item $G$ has a $2$-edge-connected subgraph with infinitely many vertices.
    \end{enumerate}
\end{theorem}
\begin{proof}
    ($\Leftarrow$) If $G$ has an infinite complete binary tree minor, then $G$ is cat-win by \Cref{thm:infinite_trees}, and if $G$ has an infinite $2$-edge-connected subgraph, then $G$ is cat-win by \Cref{lem:infConnCatWin}.
    
    ($\Rightarrow$) This is the contrapositive of \Cref{lem:infHerderWin}.
\end{proof}

\section{Cat-pseudo-win graphs}

Consider the double-sided ray, identifying vertices by $\mathbb{Z}$. Without loss of generality, the cat will play to $0$, the herder will cut edge $\{-1,0\}$, and the cat can move to $n$. The herder cuts $\{n,n+1\}$, and captures the cat in finite time. Note though, that even though the herder can guarantee isolation of the cat, the cat can score arbitrarily high. We consider this to be \emph{cat-pseudo-win}, as the herder cannot guarantee capture in a fixed number of moves. 
\begin{definition}
    If a cat has a strategy that provides a valid response move to the first $k-1$ cuts, then $G$ is \emph{$k$-evadible}. If $G$ is $k$-evadible for all $k\in\mathbb{N}$, then $G$ is \emph{$\omega$-evadible}.
\end{definition} 
To motivate this definition, we note that if the cat makes a response after the first $k-1$ cuts, then the cat can guarantee that the herder makes at least $k$ cuts in the graph. We also note that $G$ is herder-win and $\omega$-evadible if and only if $G$ is cat-pseudo-win.
\begin{lemma}\label{lem:allPaths}
    If for all $k\in\mathbb{N}$, there exists a path of length $k$ vertices ($P_k$) as a subgraph of $G$, then $G$ is $\omega$-evadible.
\end{lemma}
\begin{proof}
    If the cat is challenged to evade capture for $k$ turns, they play on a path of length $2^k$ ($P_{2^k}$), scoring at least $\lceil\log_22^k\rceil=\lceil k\rceil=k$.    
\end{proof}

\begin{lemma}\label{lem:allCycles}
    If for all $k\in \mathbb{N}$, there exists a vertex $v_k\in V(G)$ with $k$ edge-disjoint cycles containing $v_k$ in $G$, then $G$ is $\omega$-evadible.
\end{lemma}
\begin{proof}
    Suppose the cat is challenged to evade capture for $k$ turns. The cat's strategy is to choose vertex $v_k$ to start on, and always move to $v_k$ if it is legal. If the cat is already on $v_k$, they choose any cycle that is still intact that passes through $v_k$ and move to any vertex on that cycle (guaranteeing a return to $v_k$ on the next turn). How long may this strategy guarantee the cat's survival? After $i<k$ edge deletions, at least $k-i$ cycles still pass through $v_k$, and so the cat on $v_k$ may find another move. Thus this strategy guarantees survival for at least $k$ edge deletions.
\end{proof}

To show that a graph is not $\omega$-evadible, we establish that when the structures that make a graph $\omega$-evadible don't occur, it allows us to find an upper bound on the cat number.

\begin{theorem}\label{thm:boundedVictory}
    If there exists a $k$ such that no $P_k$ exists in $G$ and every vertex $v$ has fewer than $k$ edge-disjoint cycles through it, then $\cat(G)\leq k^3-2k^2+3k-2$.
\end{theorem}
\begin{proof}
    Assume for contradiction that the cat can evade capture for $k^3-2k^2+3k-2$ moves.
    
    Let the cat start at $v_0$. Suppose the cat is currently on $v_i$. The herder's strategy is to cut off all edges that are part of a cycle that passes through $v_i$. This will take at most $k(k-1)$ turns as there are fewer than $k$ edge-disjoint cycles through $v_i$, each of which has length at most $k-1$ (else there is a $P_k$). If the cat is then on $v_i$, then the herder makes any passing move, after which the cat makes a move off of $v_i$. Let $v_{i+1}$ be the first vertex that the cat moves through (possibly to) after their last visit to $v_i$. The herder's next move is to cut $v_iv_{i+1}$. At this point, the cat must not be able to reach $v_i$, else another cycle through $v_i$ must not have been cut.

    It follows that when the cat is on $v_i$, within $k^2-k+2$ cuts, the herder can force the cat to never return to $v_i$. As this play continues, the cat will create a path $v_0,\dots,v_i$. At this point, the herder has taken $(i+1)(k^2-k+2)$ cuts to force the cat off of $v_i$ forever after. For $i=k-2$, we have that the cat cannot be on any of $v_0,\dots,v_{k-2}$, and so the cat must move to $v_{k-1}$, creating a $P_k$, which must not exist. Thus the cat must be captured on $v_{k-2}$. This collectively takes at most $(k-1)(k^2-k+2)=k^3-2k^2+3k-2$ cuts.
\end{proof}

\begin{corollary}
    If the conditions of \Cref{lem:allCycles} and \Cref{lem:allPaths} are both false, then $G$ is not $\omega$-evadible.
\end{corollary}
\begin{proof}
    Assume that there is a $k_P$ such that there are no paths of length $k_P$ and a $k_C$ such that no vertex has $k_C$ disjoint cycles through it. Then the cat number is finite and bounded by  \Cref{thm:boundedVictory} with $k=\max(k_P,k_C)$.
\end{proof}

The following theorems are now self-evident.
\begin{theorem}
    $G$ is $\omega$-evadible if and only if one of the following is true
    \begin{enumerate}
        \item For all $k$, there exists a path on $k$ vertices ($P_k$) as a subgraph of $G$.
        \item For all $k$, there exists $v_k\in V(G)$ such that there exist $k$ edge-disjoint cycles, each containing $v_k$.
    \end{enumerate}
    Graphs which are $\omega$-evadible and herder-win are exactly those which are cat-pseudo-win.
\end{theorem}

Suppose that $G$ is not $\omega$-evadible. The idea for the following theorem comes from Dietmar Berwanger. We claim most graphs with finite cat number are in fact finite graphs, in the following sense.

\begin{theorem}
    If $G$ is a connected graph and not $\omega$-evadible, then $G$ is finite or $G$ contains an infinite degree vertex.
\end{theorem}
\begin{proof}
    Let $G$ be connected and $\cat(G)$ is finite. Then there exists a $k$ such that there are no $P_k$ in $G$, nor $k$ edge-disjoint cycles through any vertex. Suppose that $G$ does not contain an infinite degree vertex. Then any vertex $v$ has $d(v)<\infty$ neighbours. Let $S_i$ be the set of vertices of distance $i$ from $v$. Note that $S_i=\emptyset$ for $i> k$, since there are no $P_k$ in the graph. Further, by induction, each $S_i$ is finite, since $S_0=\{v\}$ is finite, and each vertex $v\in S_i$ can contribute at most $d(v)<\infty$ to $S_{i+1}$. Thus the set of vertices reachable from $v\in V(G)$ is finite, and so the connected graph $G$ is finite.
\end{proof}

We conclude this section by noting that this analysis generates a natural generalization of cat number for infinite graphs. We introduce ordinal cat numbers below.

\begin{definition}
    If $\alpha=\{\alpha_i|i\in {\cal I}\}$ is an  ordinal, and the cat is on vertex $v$, and after any edge cut $e$ and any $i$ selected, there exists a legal cat move in $G-e$ to $v'$ with $\cat(G-e,v')\geq \alpha_i$, then \[\cat(G,v)\geq \alpha.\]

    If $\alpha=\{\alpha_i|i\in {\cal I}\}$ is an ordinal, and the cat is on vertex $v$, and there is an edge cut $e$ such that for all legal cat responses $v'$ in $G-e$, we have that $\cat(G-e,v')\leq \alpha_i$ for some $\alpha_i\in\alpha$, then \[\cat(G,v)\leq \alpha.\]
\end{definition}

We leave a deeper analysis of infinite cat herding with this definition in mind for future work.

\section{Conclusion}
Throughout this paper, we have probed the structure of the game of Cat Herding in the extremal cases of very low cat number, as well as infinite or arbitrarily large cat number. We leave further directions on both sides. 

On the low cat number side, the next natural step is to classify all of the graphs of cat number $4$ or higher. We observe that this task may be complicated as the reductions to keep a finite classification only seem to grow in complexity as the cat number grows.

On the large cat number side, we note that transfinite cat numbers are a natural way to move the research forwards. Namely, when the cat can score arbitrarily high, but not win, we observe that our natural definition of cat number ordinals provides a way to carefully investigate this instance. In particular, these may provide the herder with a cohesive idea of how long it will be until they have captured the cat. For instance, a cat number of $\omega+4$ would mean that within $4$ herder-optimal moves, the cat must make a decision as to how much longer they will live. A cat number of $2\omega$ would mean that the cat can choose how long until they must choose how much longer they will survive. This can keep going for as many transfinite ordinals as there are.

\bibliographystyle{abbrv}
\bibliography{sample}

\begin{thebibliography}{1}

\bibitem{mengerInfinite}
R.~Aharoni.
\newblock Menger's theorem for graphs containing no infinite paths.
\newblock {\em European Journal of Combinatorics}, 4(3):201--204, 1983.

\bibitem{ashmore2024cutscatscompletegraphs}
R.~Ashmore, D.~Dyer, T.~Marbach, and R.~Milley.
\newblock Cuts, cats, and complete graphs, 2024.

\bibitem{CLARKE20121421}
N.~E. Clarke and G.~MacGillivray.
\newblock Characterizations of k-copwin graphs.
\newblock {\em Discrete Mathematics}, 312(8):1421--1425, 2012.

\bibitem{DERENIOWSKI20131950}
D.~Dereniowski, Öznur Yaşar~Diner, and D.~Dyer.
\newblock Three-fast-searchable graphs.
\newblock {\em Discrete Applied Mathematics}, 161(13):1950--1958, 2013.

\bibitem{graph_searching_complexity}
N.~Megiddo, S.~L. Hakimi, M.~R. Garey, D.~S. Johnson, and C.~H. Papadimitriou.
\newblock The complexity of searching a graph.
\newblock {\em J. ACM}, 35(1):18–44, Jan. 1988.

\end{thebibliography}
\end{document}